\documentclass{amsart}
\usepackage[all]{xy}
\SelectTips{cm}{}
\usepackage{amsmath}
\usepackage{amssymb}
\usepackage{amscd}
\usepackage{amsthm}
\usepackage{amsfonts}
\usepackage{amsxtra}
\usepackage{euscript,amstext,pb-diagram}
\usepackage{tikz-cd}[2010/10/13]
\usepackage{setspace}
\usepackage{hyphenat}

\theoremstyle{plain}

\newtheorem{Thm}{Theorem}[section]
\newtheorem{Lem}[Thm]{Lemma}
\newtheorem{Pro}[Thm]{Proposition}
\newtheorem{Cor}[Thm]{Corollary}
\newtheorem{Conj}[Thm]{Conjecture}

\theoremstyle{definition}
\newtheorem{Def}[Thm]{Definition}
\newtheorem{Exa}[Thm]{Example}

\newtheorem{Rem}[Thm]{Remark}

\newtheorem*{Rem-intro}{Remark}

\newcommand{\Dirlim}{\varinjlim}

\newcommand{\HH}{{\mathcal{H}}}

\newcommand{\LL}{{\mathcal{L}}}

\newcommand{\ZZ}{{\mathbb{Z}}}
\newcommand{\QQ}{{\mathbb{Q}}} 

\newcommand{\CC}{{\mathbb{C}}}

\newcommand{\RR}{{\mathbb{R}}}
 \newcommand{\PP}{{\mathbb{P}}}

\newcommand{\ca}{$C^*$-algebra}

\newcommand{\KK}{{\mathcal{K}}}

 \newcommand{\DD}{{\mathcal{D}}}

\newcommand{\calk}{\mathcal{Q}}

\def\:{\colon\!}
\catcode`@=11
\@xp\let\csname sub jclassname@1991\endcsname \subjclassname
\@namedef{subjclassname@2010}{%
  \textup{2010} Mathematics Subject Classification}
\catcode`@=12

\begin{document}

\title{ Spanier-Whitehead $K$-duality for $C^*$-algebras  }



 \author[Kaminker]{Jerome Kaminker}
 \address{Department of Mathematics,
  University of California,  Davis,
      Davis, CA 95616}
\email{kaminker@math.ucdavis.edu}

\author[Schochet]{Claude L.~Schochet}
\address{Department of Mathematics,
     Technion,
     Haifa 32000, Israel}

\email{clsmath@gmail.com}

\thanks{ }
\begin{abstract}

Classical Spanier-Whitehead duality was introduced for the stable
homotopy category of finite CW 
complexes. Here we provide a comprehensive treatment of a
noncommutative version, termed Spanier-Whitehead $K$-duality, which 
is defined on the category of $C^*$-algebras whose $K$-theory is
finitely generated and that satisfy the UCT, with morphisms the $KK$-groups.  We explore 
what happens when these assumptions are relaxed in various ways. In
particular, we consider the relationship between Paschke duality 
and Spanier-Whitehead $K$-duality.

\end{abstract}
\maketitle

\tableofcontents
\section{Introduction}
 Classical Spanier-Whitehead duality is a generalization of Alexander duality,
which relates the homology of a space to the cohomology of its
complement in a sphere.  Ed Spanier and
J.H.C. Whitehead \cite{SW1}, \cite{SW2},  noting that the
dimension of the sphere did not play an essential role, adapted it to
the context of stable homotopy theory. Its
 history and its relation to other classical duality ideas are described 
in depth   by Becker and Gottlieb \cite{BG}.  To be more precise,
given a finite complex $X$ there is another finite complex, the
Spanier-Whitehead dual of $X$, denoted  $DX$, 
and a duality map $\mu: X \wedge DX \to S^n $ such that slant product with
the pull-back of the generator, $\mu^*([S^n])
 \in H^*(X \wedge DX)$  induces isomorphisms 
$$\diagdown  \mu^*([S^n]) :H_*(X) \to H^*(DX).$$  Moreover, $DDX $ is stably
 homotopy equivalent to $X$. 
 Note that there is no need for any sort of
orientability requirement, in contrast  to Poincar\'e duality.
Spanier-Whitehead duality turns out to be an interesting and
fairly universal notion which generalizes to many contexts. 

Spanier-Whitehead duality extends in a natural way to generalized
cohomology theories such as K-theory.  For a finite complex $X$, a
dual finite complex $DX$ turns out to be  a K-theoretic dual as well \cite{KKS}.
Since K-theory and ordinary cohomology detect torsion differently this
result requires proof.  The essential fact is that $X \to
K^*(DX)$ defines a homology theory naturally equivalent on finite
complexes to $K_*(X)$.  We shall refer to such a dual as a
{\emph{Spanier-Whitehead $K$-dual}}.

The bivariant version of K-theory introduced by Kasparov
\cite{Kasparov} is closely related to duality.  One has, for $X$ and
$Y$ finite complexes,
\begin{equation}
  \label{eq:4}
  KK^*(C(X),C(Y)) \cong KK^*(\CC, C(DX \wedge Y)) \cong K_*(C(DX \wedge Y))  \cong K^*(DX \wedge Y)
\end{equation}
and, in fact, this can be taken as a definition of KK-theory for
finite complexes, \cite{Schochet}.

Turning to duality for $C^*$-algebras, the subject of this paper, we see that there are several 
points which must be considered:
\begin{enumerate}
\item The $C^*$-algebras  which arise naturally in applications to
  topology, dynamics, and index theory are not simply $C(X)$ for $X$ a finite 
complex. They are generally noncommutative, and the topological spaces
commonly associated with them may be completely uninteresting or intractable.
\item The cohomology theories that have been used successfully on $C^*$-algebras are
$K$-theory and its various relatives.  These do generalize topological $K$-theory but 
have less structure. There is no natural product structure when the
algebras are noncommutative, and the Adams operations 
do not extend to the noncommutative case.
\item  For a separable, nuclear $C^*$-algebra  $A$ represented on a Hilbert
  space, the commutant of its projection into the Calkin algebra
  has some of the properties of a Spanier-Whitehead $K$-dual.  This is
  the Paschke dual of $A$, which we denote  $\PP (A)$.  It satisfies
   $$K_*(\PP(A)) \cong K^*(A).$$ 
  However, in general $\PP(A)$ is not separable or nuclear,
  the Kasparov product is not defined, there is no analogous description for $K_*(A)$ 
  and one cannot simply take the Paschke dual of the Paschke dual. 
\end{enumerate}

Keeping this in mind, we shall see what can be done.  There 
are several different arenas to investigate:

\begin{enumerate}

\item If we stay within the bootstrap category \cite{Top2}  and restrict to $C^*$-algebras whose 
$K$-theory groups are finitely generated, then there is a very satisfactory duality situation. 
Spanier-Whitehead $K$-duals exist, they are suitably unique, and \lq\lq everything" works 
out as one would expect from considering the category of finite cell complexes. 

\item If we stay within the bootstrap category and allow $K_*(A)$ to be countable but not necessarily
  finitely generated, then there exist $C^*$-algebras which cannot have
Spanier-Whitehead $K$-duals for algebraic reasons. 

\item If we keep the finite generation hypothesis but no longer require the bootstrap hypothesis, then 
various things can happen, most of which are bad.

\item  We will find separable and nuclear substitutes
for the Paschke dual which may be useful in various analytic
contexts. cf. \cite{HR}
 
 \end{enumerate}

We now give a more formal summary of our results.

{\bf{Section 2}} provides the basic definitions and basic properties of Spanier-Whitehead $K$-duality.  Our purpose here is to clarify  
  the various and sometimes contradictory definitions that appear in the literature.  Our definitions require separability because we 
want the full power of the Kasparov pairing. 

In {\bf{Section 3}} we explain the relationship between classical Spanier-Whitehead duality and Spanier-Whitehead $K$-duality. 
In a word, the first implies the second for finite complexes, but this  is not automatic from the axioms; it requires a spectral sequence 
comparison theorem that we established many years ago.  

Spanier-Whitehead $K$-duality arises in several different areas of
mathematics. {\bf{Section 4}} discusses how this type of duality
arises naturally even when the algebras are simple, hence are very far
from commutative ones.  We discuss examples drawn from hyperbolic
dynamics, the Baum-Connes conjecture, and others.

In {\bf{Section 5}} we start a discussion of the relationship of
Poincar\'e duality as used in noncommutative geometry and the
traditional notion from topology.

{\bf{Section 6}}  is devoted  to establishing a very important and basic result.  Every separable nuclear $C^*$-algebra in the bootstrap 
category with finitely generated $K$-theory groups has a Spanier-Whitehead $K$-dual 
that is suitably unique.
   We show further that this is the largest category of $C^*$-algebras with 
this property.

Then comes the bad news.   In {\bf{Section 7}} we give a concrete example of a $C^*$-algebra that is separable, nuclear, bootstrap and yet has no Spanier-Whitehead 
$K$-dual.  The example we provide is pretty basic: it is an AF-algebra
with $K_0 = \QQ $. 

{\bf{Section 8}} provides an interesting application of the theory to
mod-p K-theory.  Indeed, these issues led the second author to
initiate the current study.

{\bf{Section 9}} is devoted to Paschke duality.  We show how this differs in basic ways from Spanier-Whitehead $K$-duality  but resembles it in other ways. 
The main problem is that the Paschke dual is typically not separable or nuclear.  {\bf{Sections 10 and 11}}  develop some tools to help us replace non-separable, 
non-nuclear $C^*$-algebras with smaller versions of themselves.  We are motivated morally (though not at all in a technical sense) by the fact that 
any topological space is weakly equivalent to a CW-complex.

In a future paper we  will see what can be done when separability and nuclearity are not assumed. 
We have in mind 
the possibility of replacing $KK$ by the Brown-Douglas-Fillmore $Ext$ groups, which agree 
with the $KK$ groups when $A$ is separable nuclear.  A generalization
of Spanier-Whitehead $K$-duality would be very useful here and could
yield insight on  the following conjecture.

 Suppose that $A$ is a separable $C^*$-algebra in the bootstrap category with $K_*(A)$ finitely generated.  Then it has a Spanier-Whitehead 
$K$-dual $DA$   and it also has a Paschke dual
$\PP (A)$.  The UCT implies that there is an element   $u \in KK_0(DA, \PP(A)) $ inducing an isomorphism 
\[
u_* : K_*(DA) \overset{\cong}\longrightarrow K_*(\PP(A)).
\]
We may regard $u \in  Ext(DA, S\PP(A))$  since $DA$ is separable nuclear.

\begin{Conj}  
There exists an element $v   \in  Ext(S\PP(A), DA ) $  and enough of the $KK$-pairing transfers over to $Ext$ 
so that one can say that $DA$ and  $S\PP(A)$ are $\lq\lq Ext  $
equivalent\rq\rq\,\, via the duality classes $u$ and $v$, in a suitable categorical setting.
\end{Conj}

 \vspace {.3in}

{\emph{Some technical notes:}}
\begin{enumerate}
\item Signs:  In the classical Spanier-Whitehead duality pairing $X \wedge DX \to S^n$, the 
number $n$ is determined by the dimension of the sphere in which $X$ is initially embedded. 
It is thus not intrinsic to the problem. It does, however, control the shift in dimension that occurs when 
passing from the homology of $X$ to the cohomology of $DX$ and hence $DX$ is frequently 
denoted $D_nX$ or $D_{n-1}X$.  Working in periodic $K$-theory the number is even 
less 
important, since all that matters is its parity. The result is that either the  duality classes $\mu $ and $\nu $ both appear in 
$KK_1$ or both appear in $KK_0$.  In the case of $KK_0$ no attention to signs is required. In the 
case of $KK_1$ (and this is the case in the paper of  Putnam-Kaminker-Whittaker \cite{KPW}, for example) there are various
changes in sign forced by the parity requirement. We will stay away from this case for simplicity, confident that the reader can 
see the necessary changes needed from the Putnam-Kaminker-Whittaker paper.

\item When we say that \lq\lq $A$ satisfies 
the UCT" we mean that for all $C^*$-algebras $B$ with countable approximate unit,
the Kasparov groups $KK_*(A,B)$ satisfy the Universal Coefficient Theorem \cite{RS}. We
conjectured at the Kingston conference (1980)  that every separable nuclear $C^*$-algebra was equivalent to a 
$C^*$-algebra in the bootstrap category \cite{Top2} and hence satisfied the UCT; this 
conjecture is still open and more plausible than ever.

\item The analogy between the stable homotopy category and the
  category of $C^*$-algebras with $KK$-theory as morphisms has been
  developed by several people, e.g. \cite{Meyer, Meyer-Nest}.  In that
  context Spanier-Whitehead K-duality and classical Spanier-Whitehead
  duality arise in similar ways \cite{Meyer,BG}.  The fact that there are geometric
  and dynamical instances in the noncommutative setting perhaps
  enhances their interest.  Nevertheless, we will not develop this
  aspect in the present paper.

  \end{enumerate}

It is a pleasure to acknowledge assistance from Heath Emerson, Peter Landweber, 
  Lenny Makar-Limonov, Orr Shalit, and Baruch Solel
  in the creation of this article. Special thanks go to Ilan Hirshberg
  and Jonathan Rosenberg for their substantial contributions to Section
  10. A special thanks goes to the referee for a meticulous and very
  helpful report.  Claude Schochet is  also very conscious of a bridge game
that he played as a graduate student in 1968 with Ed Spanier, G.W.  Whitehead, and N.E. Steenrod, sitting in for his advisor Peter May and thinking that his whole mathematical career 
was on the line.

\vglue .5in

\section{Spanier-Whitehead $K$- Duality}

The existing literature is somewhat confused regarding the proper definition of 
Spanier-Whitehead $K$-duality.  The basic idea is natural and seems to have appeared first in
\cite{Kasparov}.  Connes considered a noncommutative version of
Poincar\'e duality which refers to algebras dual to their opposite
algebras, but some of his examples have the important additional structure of a
fundamental class, which
make them especially interesting.  They were precursors to his notion of
spectral triple as a noncommutative manifold.  Basically, he proved
that the existence of the fundamental class yielded what we are
calling Spanier-Whitehead K-duality classes.  In \cite{KP} Kaminker
and Putnam referred to it as Spanier-Whitehead duality explicitly.
The definitions used are essentially the same, but there are some
technical points which we will clarify in this section.

\begin{Def}  
Suppose given separable $C^*$-algebras $A$ and $DA$ together 
with $KK$-classes 
 
\[
\mu \in KK_*(\CC, A\otimes DA).
\]
\[
\nu \in KK_*(A\otimes DA , \CC )
\]
with the property that 
\[
\mu \otimes _A \nu = \pm 1_{DA}  \in KK_0(DA, DA)
\]
and 
\[
\mu \otimes _{DA} \nu = \pm 1_{A}  \in KK_0(A, A).
\]
Then $A$ and $DA$ are said to be {\emph{Spanier-Whitehead $K$-dual}}
with duality 
classes $\mu$ and $\nu $. 
\end{Def}

The separability condition is to ensure that the $KK$-products 
are defined. (We discuss weakening this condition later in the paper.) 
Note that this definition is symmetric.    If both classes have even parity then the sign is $+1$ in both 
cases; in the odd case one introduces signs as in \cite{KPW,E1}.

\begin{Thm}\label{twoduals}
\begin{enumerate}
 \item 
Suppose given Spanier-Whitehead $K$-dual $C^*$-algebras $A$ and $DA$. 
Then each of the associated slant product maps

\begin{itemize}
\item 
\[
(-)\otimes_A  \nu : K_*(A) \longrightarrow K^*(DA)
\]
\item 
\[
(-)\otimes _{DA} \nu : K_*(DA) \longrightarrow K^*(A)
\]
\item
\[
 \mu \otimes _{DA}(-)   : K^*(DA) \longrightarrow K_*(A)
\]
\item
\[
  \mu \otimes _A (-)  : K^*(A) \longrightarrow K_*(DA)
\]
\end{itemize}
is an isomorphism, and the compositions 
\[
K_*(A) \longrightarrow K^*(DA)  \longrightarrow  K_*(A)  
\]
and
\[
  K_*(DA)   \longrightarrow K^*(A) \longrightarrow  K_*(DA)
\]
  are each $\pm 1$.  

\item Conversely, given separable $C^*$-algebras $A$ and $DA$ together with classes 
$\mu$ and $\nu$,   if the indicated compositions are $\pm 1$ then $A$ and $DA$ are Spanier-Whitehead $K$- dual.
\end{enumerate}

\end{Thm}
 
\begin{proof}  (Although versions of this appear in the literature, we
  include a proof for completeness.)

We are given duality classes 
\[
\mu \in KK_0(\CC, A\otimes DA)    \qquad\qquad \nu \in KK_0(A\otimes DA, \CC)
\]
so that we are considering the case where no signs appear.
Instead of assuming that $x \in KK_0(DA, \CC )$, which would suffice for the first part 
of the proof, we assume that we are given auxiliary separable $C^*$-algebras $F$ and $G$ and that
\[
x \in KK_0(DA\otimes F, \CC \otimes G).
\]

We shall prove that the composite map 
\begin{equation}
\label{thing}
KK_0(DA\otimes F, \CC \otimes G)
 \xrightarrow{\mu  \otimes _{DA} (-)}
KK_0(\CC\otimes F, A\otimes G)
\xrightarrow{(-)\otimes _A \nu }   KK_0(DA\otimes F, \CC \otimes G)
\end{equation}
is the identity map. By symmetry, it follows that the dual composite map 
\[
KK_0(\CC \otimes F, A \otimes G) \to KK_0(\CC\otimes F, A\otimes G) 
\]
is an isomorphism and this proves the proposition. 

 Let $1_A \in KK_0(A,A) $ denote the class of the identity map, and 
then let 
\[
1_A \otimes w = 1_A \otimes _{\CC} w  \in KK_0(A\otimes Y, A\otimes Z) 
\]
 denote the external product of $1_A$ with 
some class  $w \in KK_0(Y, Z)$.  Then:

\[
(\mu \otimes _{DA} x)\otimes _A \nu   =
\]
\[
=  \big(\mu \otimes _{A\otimes DA} (1_A \otimes x)\big) \otimes _A \nu
\]
\[
= \big[ \big(\mu \otimes _{A\otimes DA} (1_A \otimes x)\big)\otimes 1_{DA} \big]\otimes _{A\otimes DA} \nu
\]
 \[
=\big[ \mu \otimes _{A\otimes DA} \big( 1_A \otimes x \otimes 1_{DA} \big)\big] \otimes _{A\otimes DA} \nu
\]
 \[
=\mu \otimes _{A\otimes DA} \big[\big( 1_A \otimes x \otimes 1_{DA} \big)\otimes _{A\otimes DA} \nu\big]
\]
(because $\otimes _{A\otimes DA}$ is associative)
\[
= \mu \otimes _{A\otimes DA} \big( x \otimes _{\CC} \nu \big)
\]
 \[
= \mu \otimes _{A\otimes DA} \big( \nu  \otimes _{\CC} x \big)
\]
(because $\otimes _{\CC }$ is commutative)
\[
= ( \mu \otimes _A \nu) \otimes _{DA} x  = 1_{DA} \otimes _{DA} x = x  .
\]

Conversely, suppose that the composite (\ref{thing})
\[
KK_*(DA\otimes F,\CC\otimes G ) \to KK_*(DA\otimes F, \CC\otimes G  ) 
\]
 is the identity. 
This translates into the formula
 \[
(\mu \otimes _{DA} x)\otimes _A \nu   =   x
\]
for all $x \in KK_0(DA\otimes F, \CC\otimes G )$.  Set $F = \CC $,  $G = DA$  and $x = 1_{DA}$. 
Then we have 
\[
1_{DA} =   (\mu \otimes _{DA} 1_{DA} )\otimes _A \nu   = \mu \otimes _{DA} \nu
\]
as desired. By symmetry, 
\[
1_A = \nu \otimes _A \mu 
\]
and the proof is complete.

\end{proof}

\begin{Cor}  Suppose given two pairs of Spanier-Whitehead dual algebras $A$ and $DA$ and also $B$ and $DB$ with associated duality classes 
$\mu_A$, $\nu_A $, $\mu_B$, $\nu_B$.  Then these classes determine canonical isomorphisms
\[
KK_*(A, B) \cong KK_*(DB, DA ).
\]
\end{Cor}

\begin{proof}  The natural map 

\[
KK_*(A,B)  \xrightarrow{(-)\otimes _B \nu _B }   KK_*(A\otimes DB, \CC )  \xrightarrow{ \mu_A \otimes _A (-) }  KK_*(DB, DA)
\]
\vglue .2in
\flushleft is obtained by taking special cases of equation (\ref{thing}) and its dual. 
\end{proof} 
  
\begin{Cor}
Suppose with the notation above that we are given 
 
\[
\mu \otimes _A \nu = u  \in KK_0(DA, DA)
\]
and 
\[
\mu \otimes _{DA} \nu = v  \in KK_0(A, A)
\]
where $u$ and $v$ are $KK$-invertible elements, not necessarily $\pm 1$.  
Then the four slant products listed in Theorem  \ref{twoduals} will be isomorphisms, and 
the composites
\[
K_*(A) \to K_*(A)
\]
and
\[
K_*(DA) \to K_*(DA)
\]
will be the isomorphisms  $v_*$ and $u_*$ respectively. 

Conversely,  if $A$ and $DA$ satisfy the UCT
  and the composites $u_*$ and $v_*$ are isomorphisms then $u$ and 
$v$ are $KK$-invertible.
\end{Cor}

 \begin{proof} This is mostly immediate from the Theorem. The missing link is 
provided by the following proposition, which is of independent interest.
\end{proof}

\begin{Pro} Suppose that $A$ is a $C^*$-algebra satisfying the UCT  and 
there is an element $u \in KK_0(A,A)$ such that 
\[
u_* : K_*(A) \longrightarrow K_*(A)
\]
is an isomorphism. Then $u$ is $KK$-invertible.  If $u_* = \pm 1$ then $u = \pm 1 + k$ 
for some $k \in Ker(\gamma _\infty )$, where
\[
\gamma _\infty : KK_*(A,A) \longrightarrow Hom (K_*(A), K_*(A))
\]
is the index map in the UCT.

\end{Pro}

\begin{proof}  (Thanks to L. Makar-Limanov for help with this proof.)  The UCT sequence
\[
0 \to Ker(\gamma _\infty ) \longrightarrow  KK_*(A,A) \overset{\gamma _\infty}\longrightarrow 
End(K_*(A)) \to 0
\]
 splits as rings, and $Ker(\gamma _\infty)^2 = 0$, \cite{RS}.  Write
\[
u = w + k
\]
for some $k \in Ker(\gamma _\infty)$ and  $w \in KK_*(A,A)$
an invertible element coming from an invertible in $End(K_*(A))$ via
the splitting.  (If 
$u_* = \pm 1 $ then  $w = \pm 1 $.) 
 Write $x = w^{-1}.$
Then
\[
u =  w + k = (wx)(w + k) = w(1 + xk)
\]
and hence
\[
\big[(1 - xk)x\big] u = (1 - xk)xw(1 + xk) =  1 - (xk)^2 = 1
\]
since $(xk)^2 = 0$.  Thus $u$ is $KK$-invertible.

\end{proof}

 Here are some basic properties of Spanier-Whitehead $K$-duality, cf.\cite{KP}. 

\begin{Thm} \label{T:basic}
Suppose that $A$ and $DA$ are Spanier-Whitehead $K$- dual and both satisfy the UCT. 
Then:
\begin{enumerate}
\item  $D(DA)$ is $KK$-equivalent to $A$. 
\item  $K_*(A)$ is finitely generated. 
\item If $Q$ and $R$  satisfy the UCT then slant pairing 
with the Spanier-Whitehead $K$-duality  classes yield natural inverse isomorphisms
\[
\mu _* : KK_*( Q\otimes DA, R) \overset{\cong}\longrightarrow KK_*(Q, R\otimes A)
\]
\[
\nu _* : KK_*(Q, R\otimes A)    \overset{\cong}\longrightarrow     KK_*( Q\otimes DA, R) 
\]
\item If   $A$  is $KK$-equivalent to $B$, then 
$B$ has a Spanier-Whitehead $K$- dual $DB$ and $DA$ is $KK$-equivalent to $DB$.
\end{enumerate}
\end{Thm}  \qed

\begin{Thm} Suppose that $A$, $B$, and $A\otimes B$ each have Spanier-Whitehead $K$-duals.   Then there is a natural $KK$-equivalence 
\[
D(A\otimes B) \simeq  DA \otimes DB 
\]
\end{Thm}
  \begin{proof} 
Under   natural duality class maps  (which are isomorphisms by (\ref{thing}))
\vglue .1in
\[
  KK_0(D(A\otimes B) \otimes A \otimes B , \CC)  \cong  
KK_0(D(A\otimes B) \otimes A  , DB) \cong
\]
\[
\cong    KK_0(D(A\otimes B),     DA \otimes DB)
\]
\vglue .1in
\flushleft the class
$\nu_{A\otimes B}   \in KK_0(D(A\otimes B) \otimes A \otimes B , \CC) $ is sent to a class which we designate
\[
\Psi \in   KK_0(D(A\otimes B),     DA \otimes DB))  .
\]
We can similarly produce a class   $\Phi \in KK_0(DA \otimes DB, D(A\otimes B))$ simply by using (\ref{thing}) and its dual a few times.  Then a proof similar 
to the proof of Theorem 2.2          shows that $\Phi = \Psi^{-1}$. 
\end{proof}

\section{Fitting Classical Spanier-Whitehead duality into the Spanier-Whitehead $K$-duality framework}

  Classical Spanier-Whitehead duality actually lives in the world of stable homotopy 
theory. Thus its beautiful properties need some modification before  the relationship 
with Spanier-Whitehead $K$-duality emerges.

We borrow the following exposition from Becker-Gottlieb \cite{BG}, \S
4.  Given a polyhedron $X$ in $S^{n+1}$, Spanier-Whitehead define an
$n\mbox{-} dual$, $D_nX$, to be a polyhedron contained
in $S^{n+1} - X$ which has the property that some suspension of $D_nX$
is a deformation retract of the corresponding suspension of $S^{n+1} -
X$.  

 Now suppose that $X^* \subset (S^{n+1} - X)$ is a polyhedron which is actually a deformation retract,
 hence an n\hyp dual. 
Following Spanier, remove a point of $S^{n+1} $ that is neither in $X$ nor in $X^*$. Then one can regard both spaces as embedded in $\RR ^{n+1}$.   Define 
\[
\mu^X : X \times X^* \longrightarrow S^{n}
\]
by
\[
\mu ^X (x, x^*) = (x - x^*)/|x - x^*|.
\]
The restriction of $\mu ^X$ to $X \vee X^*$ is null-homotopic and so one obtains a map 
\[
\mu^X : X \wedge X^* \longrightarrow S^{n}.
\]
Slant product with this class induces an isomorphism  
\[
\mu ^X/(-)^*  :  H_q(X) \overset\cong\longrightarrow  H^{n-q}(X^*).
\]
Spanier, following work of Wall, Freyd, and Husemoller (see \cite{BG} for details and references)  shows that the whole 
duality theory can be expressed in terms of the duality map $\mu $.   
The space $X^*$ depends upon the choice of $n$, the choice of embedding, and the choice of 
the deformation retraction.
  It turns out, though,  that for $n$ large the stable homotopy type
  of $X^*$ is independent, up to suspension, of the choices of the embedding and of the deformation retraction.  The resulting (stable) space $DX$,
defined for any finite complex $X$, is called the {\emph{Spanier-Whitehead 
dual}} of $X$, in honor of the people who discovered it and determined its primary properties \cite{SW1}, \cite{SW2}. 
Taking $n$ large enough to be in the stable range  we have a duality 
pairing as
\[
\mu _{C(X)} : C(S^{2n}) \longrightarrow C(X) \otimes C(DX).
\]
We use even-dimensional spheres to control the parity of the degree of
the duality class.  The associated  candidate for a duality class 
\[
\nu _{C(X)} : C(X) \otimes C(DX) \longrightarrow C(S^{2k})
\]
may be obtained by taking the stable dual $\nu ^X$ of the map $ \mu ^X : X \times DX \longrightarrow S^{2k}$.
 
\begin{Thm} \label{oldKS}
\label{classical=KK}
Suppose that $X$ is a finite CW complex and that $DX$ is a Spanier-Whitehead
dual for $X$.  Then $C(X)$ and $C(DX)$ are  Spanier-Whitehead $K$- dual. Indeed,
\[
D(C(X)) \cong C(DX) .
\]
 
\end{Thm}

\begin{proof}  
This result is non-trivial, since an algebraic  isomorphism in homology does not imply 
an isomorphism in $K$-theory. However, this result was established
previously with D.S. Kahn in \cite{KKS}. It was shown there that it
follows from the identification of $K_*(X) $ with $K^*(DX)$ as discussed in the introduction.

\end{proof}

\begin{Exa}  When defining duality, one might be tempted to always require 
that the classes $\mu $ and $\nu $ actually be $KK$-inverses of one another. 
Here is an example to show that  this is a bad idea.

 Suppose that $X$ is a mod $p$ Moore space. That is, its reduced homology is zero except
in degree one, and $H_1(X;\ZZ ) \cong \ZZ/p$.   This space is
self-dual in the classical Spanier-Whitehead sense.  In fact, the
reduced cohomology of $X$ is zero except in dimension two, and
$H_1(X,\ZZ) \cong H^2(X, \ZZ)$.

There are stable duality maps $S^r \to X \wedge X \to S^t $ such that slant product with these
maps yields isomorphisms in reduced homology and cohomology
\[
H_*(X) \cong H^*(X) \qquad\qquad   H^*(X) \cong H_*(X)
 \]
with degree shifts.   However, the composite 
\[
H_*(X\wedge X ) \longrightarrow H_*(S^*) \longrightarrow H_*(X \wedge X)
\]
cannot possibly be the identity map, since $H_*(X \wedge X)$ has torsion and $H_*(S^*)$ is torsionfree.

Write $A = C_0(X - pt)$ so that 
\[
K_0 (A) = \ZZ/p\qquad\qquad K_1(A) = 0
\]
Then the K\"unneth Theorem \cite{Top2} implies that there are isomorphisms
\[
K_0(A)\otimes K_0(A) \overset{\cong}\longrightarrow K_0(A \otimes A)
\]
and
\[
K_1(A \otimes A) \overset{\cong}\longrightarrow Tor(K_0(A),K_0(A))
\]
so that
\[
K_0(A \otimes A) = \ZZ/p  \qquad\qquad  K_1(A\otimes A) = \ZZ/p
\]
and the UCT \cite{RS} implies that 
\[
KK_0(A\otimes A, \CC ) = \ZZ/p  \qquad   KK_0(\CC, A\otimes A ) = \ZZ/p
\]
The resulting pairing 
\[
KK_0(\CC, A\otimes A)  \times  KK_0(A\otimes A, \CC ) \xrightarrow{\otimes_{A\otimes A}} KK_0(\CC, \CC )
\]
is evidently trivial since $KK_0(A\otimes A, \CC )$ and $KK_0(\CC,
A\otimes A )$ are both torsion groups, while $KK_0(\CC, \CC) \cong \ZZ$. 
Thus the classical Spanier-Whitehead duality classes \cite{SW1}, \cite{SW2}
give us $K$-duality classes but 
 do NOT give us invertible $KK$-classes.

\end{Exa}


\section{Examples of noncommutative duality}

The results of the previous section seem to suggest, at least when
K-theory is finitely generated, that Spanier-Whitehead K-duality is a commutative
phenomenon.  However,  many of the algebras providing natural examples of duality 
owe this property to underlying  geometry and dynamics and are very
far from being commutative.  Indeed, many are simple algebras.  We will survey some of these in this section.

The importance of finite complexes in algebraic topology is the fact that they are constructed systematically out of basic
building blocks which are determined by their homology, e.g. spheres.  This
information can be assembled to compute 
homology and cohomology for general finite complexes.  

In the noncommutative case one is often confronted with simple algebras,
i.e. ones with no nontrivial ideals.  It is natural to look for
building blocks in $\mathcal{KK}_F$ which are of this type and, because of the
results above, one may choose to consider simple $C^*$-algebras which have
Spanier-Whitehead K-duals.  We will discuss two examples of this phenomenon---the first
coming from the study of hyperbolic dynamics and the second from the
study of hyperbolic groups.  We will then briefly consider additional
instances of noncommutative duality.

\subsection{Hyperbolic dynamics}

We refer to \cite{KPW} for precise statements and details.  A Smale
space is a compact metric space, $X$, along with an expansive
homeomorphism, $\phi$, which
has similar properties to that of an Anosov diffeomeorphism of a torus.  By
this we mean that there are two equivalence relations defined on $X$
called stable and unstable equivalence.  Each defines a locally
compact groupoid with Haar system and hence one may associate
$C^*$-algebras to them.  Let us denote them by $\mathcal{S}$ and $\mathcal{U}$.  Both
can be represented on $L^2(X)$ and the groupoids can be viewed as
``transverse'' because each stable equivalence class meets an
unstable class in a countable set.  This implies  that the product of an element of $\mathcal{S}$ and an
element of $\mathcal{U}$ is a compact operator. 

Using the automorphisms
induced by $\phi$ on $\mathcal{S}$ and $\mathcal{U}$ one constructs the crossed product
algebras, $\mathcal{R}^u = \mathcal{U} \rtimes_{\phi} \ZZ$ and $\mathcal{R}^s = \mathcal{S} \rtimes_{\phi} \ZZ$, called Ruelle
algebras.  They can be shown to be Spanier-Whitehead K-dual, \cite{KPW}.  It is interesting to
consider the construction of the duality classes. One  first obtains a
projection in $\mathcal{S} \otimes \mathcal{U}$ and from that a unitary in $\mathcal{R}^u \otimes \mathcal{R}^s$
which yields a class 
$\delta \in KK^1(\CC, \mathcal{R}^u \otimes \mathcal{R}^s)$. Then, strongly using the hyperbolic properties of
the dynamics, one constructs an extension which yields an element 
$\Delta \in KK^1(\mathcal{R}^u \otimes \mathcal{R}^s,\CC)$.  These classes are the required
duality classes.

An example of a Smale space is a subshift of finite type associated to
a matrix $A$.  Associated to this data are the Cuntz-Krieger algebras
$O_A$ and $O_{A^T}$.  It turns out that the Ruelle algebras
$\mathcal{R}^u$ and $\mathcal{R}^s$ are isomorphic to $O_A \otimes
\mathcal{K}$ and $O_{A^T} \otimes \mathcal{K}$, and so the Cuntz-Krieger algebras $O_A$ and $O_{A^T}$ are (stably) Spanier-Whitehead $K$-dual. 

\subsection{Baum-Connes conjecture}

Let $\Gamma$ be a torsion free and non-elementary Gromov hyperbolic group.  It has been shown by de la
Harpe \cite{de la Harpe} that $C^*_r(\Gamma)$ is a simple $C^*$-algebra. We will assume that there
is a model for the classifying space  $B\Gamma$ which is a closed smooth manifold.  The
Baum-Connes conjecture, which is known to hold in this case \cite{Mineyev-Yu},   asserts that there is an isomorphism,
\begin{equation}
  \label{eq:1}
  \mu: KK(C(B\Gamma), \CC) \to KK(\CC, C^*_r(\Gamma)).
\end{equation}
In the present setting the map $\mu$ can be obtained via Kasparov product
with the class in $\Psi_\Gamma \in KK(\CC,C^*_r(\Gamma) \otimes C(B\Gamma) )$ determined by the
Mishchenko line bundle,
\begin{equation}
  \label{eq:2}
  C^*_r(\Gamma) \to E\Gamma \times_\Gamma C^*_r(\Gamma) \to B\Gamma.
\end{equation}
This is the first duality class $\mu$.  As in the dynamical situation above,
very little special structure is needed to define it.  However, as
above, the other duality class $\nu$ makes use of the 
hyperbolic structure of the group.  That class is the dual-Dirac class
\begin{equation}
  \label{eq:3}
  \kappa_\Gamma \in KK(C^*_r(\Gamma) \otimes C(B\Gamma), \CC )
\end{equation}
introduced by Kasparov.  Thus, in this context, the Baum-Connes
conjecture is the same as $C(B\Gamma)$ being Spanier-Whitehead K-dual
to $C_r^*(\Gamma)$.

There is a possible connection between these examples.  The hyperbolic group
$\Gamma$ acts amenably on its Gromov boundary, $\partial \Gamma$.  If
we choose a quasi-invariant measure on $\Gamma$ then, by a result of
Connes, Feldman, Weiss \cite{Connes-Feldman-Weiss} that action is
orbit equivalent to a $\ZZ$ action.  Although this result is in a measure
theoretic setting,  in certain cases, such as a Fuchsian group of the
first kind 
acting on $S^1$ \cite{Bowen-Series}, the transformation generating
the $\ZZ$ action can be taken to be a piecewise homeomorphism which 
can be studied using hyperbolic dynamics.  Indeed, both of the
$C^*$-algebras associated to this hyperbolic dynamical system in
the first example are isomorphic to the crossed product, $C(\partial
\Gamma) \rtimes \Gamma$ \cite{Spielberg}.  This has been generalized to $SL(2,\ZZ)$ acting on
$S^1$ \cite{Laca-Spielberg}  but in this case the isomorphism
between the dynamical algebras and the crossed products is obtained by
computing K-theory and applying the classification result of Kirchberg
and Phillips.  Duality in general for hyperbolic groups acting on
their boundary has been studied in detail by Emerson \cite{E1}.  This suggests the question of whether the proof of the
Connes-Feldman-Weiss theorem, in the case of a hyperbolic group acting
on its boundary, can be refined so that one obtains a hyperbolic
dynamical system for which the associated Ruelle algebras are
isomorphic to the crossed product.  

A general theory of duality on the level of groupoids with hyperbolic
structure has been developed by Nekrashevych \cite{Nekrashevych}.
There is a setting in which analogs of the stable and unstable
groupoids can be defined, but as of yet there is no general K-theory
result involving the associated $C^*$-algebras.  It would be
interesting to show that they are Spanier-Whitehead K-dual.

\subsection{Mukai transform}
The actual Mukai transform is studied in the context of algebraic
geometry and relates the derived category of coherent sheaves on an
abelian variety to that of its dual variety \cite{Mukai}.
However, the formula for the transform can be identified with the map
in the Baum-Connes example above, and hence can be viewed as an instance of
Spanier-Whitehead K-duality.  The fact that it yields an isomorphism
was first proved by Lusztig \cite{Lusztig}, which was one of the
origins of the K-theoretic approach to such problems. We mention this
here because it indicates the sense that this type of duality is like a
``transform''.

Let $\Lambda \subseteq
\RR^n$ be a lattice and $T^n = \RR^n / \Lambda$ the associated torus.
Let $\hat T^n = \hat \RR^n / \hat \Lambda$ be the dual torus, where
$\hat \Lambda = \{\alpha \in \hat \RR^n | \alpha(x) \in \ZZ, \ \text{for}\
x \in
\Lambda \}$. The Poincar\'e line bundle, $\mathcal{P}^{\Lambda}$, over $T^n \times \hat T^n$ is
determined by the property that $\mathcal{P}^{\Lambda} |_{T^n \times \alpha} = L_\alpha$,
where $L_\alpha =  L_{\alpha_{1}} \otimes \ldots \otimes
L_{\alpha_{n}}  $.  The Mukai transform is obtained as

\begin{equation}
 \begin{tikzcd}
  K^*(T^n) \arrow{r}{p_{T^n}^*} &K^*(T^n \times \hat T^n) \arrow{r}{\mathcal{P}^{\Lambda} \otimes} & K^*(T^n \times \hat
  T^n) \arrow{r}{(p_{\hat T^n})_{!} } & K^*(\hat T^n)
 \end{tikzcd}
\end{equation}

We also have the Mishchenko line bundle, $C^*(\Lambda) \to \Psi^{\Lambda} \to
T^n$.  There is a map induced by the Gelfand transform 
\begin{equation}
 1 \otimes G : KK(\CC, C(T^n) \otimes C^*(\Lambda)) \to KK(\CC, C(T^n) \otimes
 C(\hat T^n) ), 
\end{equation}
with the property that $1 \otimes G ([\Psi^{\Lambda}]) =
[\mathcal{P}^{\Lambda}]$.  The diagram below expresses the relation between the
Baum-Connes map and the Mukai transform in this setting. We assume $n$
is even to simplify the diagram.


 \begin{equation}
\begin{tikzcd}[column sep=small, row sep=huge]
& KK(\CC, C(T^n) \otimes C(\hat T^n)) \arrow{rd}{(p_{\hat T^n})_{!} \circ
(\mathcal{P}^{\Lambda} \otimes )} & \\
KK(\CC, C(T^n)) \arrow{d}{PD} \arrow{ru}{p_{T^n}^*} \arrow{rr}{\text{Mukai transform}} & & KK(\CC, C(\hat T^n) )\\
KK(C(T^n), \CC) \arrow{rr}[swap]{\text{Baum-Connes map}}{\Psi^{\Lambda} \otimes_{C(T^n)}} & & KK(\CC, C^*(\Lambda)) \arrow{u}{G}
\end{tikzcd} 
\end{equation}

\section{Poincar\'e duality}
We will assume in this section that our algebras are unital and are in
$\mathcal{KK}_F$.  We also avoid formulating statements for odd
Poincar\'e duality.

In \cite{Connes} Connes (see also \cite{Connes-Skandalis, Kasparov}) discussed a notion of Poincar\'e duality for
a \ca.   It 
states that an algebra $A$ satisfies Poincar\'e
duality if it is Spanier-Whitehead K-dual to its opposite algebra,
$A^{op}$. 
 This yields a class
$\partial \in KK(A, \CC)$ by setting $\partial = 1
\otimes_{A^{op}} \mu $, where $1 \in KK(\CC, A^{op})$ and $\mu \in
KK(A^{op} \otimes A, \CC)$ is the duality class.  In the commutative
case $\partial$ would correspond to a K-theory fundamental class and
taking cap product with it would yield an isomorphism $$\cap \partial  :
KK(\CC, A) \to KK(A, \CC).$$
Since such an algebra $A$ is Morita equivalent to its opposite, we may
just as well formulate Poincar\'e duality in terms of $A$ alone.

If $A$ is not commutative there is, in general, no
cap product in K-theory.  We will present here a slightly weaker condition which will allow a 
version of a cap product to be defined so that one could obtain a Poincar\'e duality isomorphism of
the usual form. Note that we are using the convention that $1 \in KK(\CC,A)$ is the class of the
identity element in $A$, while $1_A \in KK(A,A)$ is the class of the
identity homomorphism.

 Let $$\tau^A : KK(B,D) \to KK(B \otimes A, D \otimes A)$$ and
$$\tau_A : KK(B,D) \to KK(A \otimes B, A \otimes D)$$ denote the
standard homomorphisms.
\begin{Def}
A $C^*$-algebra $A$ is {\emph{K\mbox{-}commutative}} if there is a class \linebreak  $m \in KK(A\otimes A, A)$ with the
 property that one has
\begin{equation} 
\tau_A \otimes_{A\otimes A} m =1_A, \ \  
\tau^A \otimes_{A\otimes A} m = 1_A.
\end{equation}
\end{Def}

Recall that, when $A$ is commutative, $m$ plays the role of the class
determined by the diagonal map and it also agrees with the class determined by the multiplication in
$A$.  We will call $m$ a $K\mbox{-}commutative\  product$.  If such a class
exists one defines the usual cup and cap products via the following diagrams.

Cup product:

\begin{equation}
  \label{eq:1}
  \begin{tikzcd}
KK(\CC,A) \times KK(\CC,A) \arrow{r}{\cup} \arrow{d}{\otimes} & KK(\CC,A)\\
KK(\CC,A \otimes A) \arrow {ur}[swap]{( - )\otimes_{A \otimes A} m}
  \end{tikzcd}
\end{equation}

Cap product:

\begin{equation}
  \label{eq:3}
  \begin{tikzcd}
KK(\CC,A) \times KK(A,\CC) \arrow{r}{\cap} \arrow{d}[swap]{(id, m
  \otimes_{A}(- ))} & KK(A,\CC) \\
KK(\CC,A) \times KK(A \otimes A, \CC) \arrow{ur}[swap]{\otimes_A}
 \end{tikzcd}
\end{equation}

\begin{Def}
Let $A$ be an algebra with a K-commutative product.  A ${fundamental\  class}$ is an 
element $\partial \in KK(A, \CC)$  such that $$\cap\ \partial: KK(\CC, A) \to KK(A, \CC)$$ is
an isomorphism.
\end{Def}
\begin{Pro}

Let $A$ be a K-commutative algebra satisfying Poincar\'e duality with
duality classes $\nu$ and $\mu$.  Then for any $u \in
KK(\CC, A)$ which is invertible with respect to cup product, the class $u \otimes_A \nu $ is a fundamental class.

\end{Pro}
\begin{proof}
We must show that if $x \in KK(\CC,A)$ has an inverse with respect to
cup product then the map $x \longmapsto x
\cap (u \otimes_A \nu) $ is an isomorphism.  Unraveling the definitions and
using properties of the Kasparov product as in Theorem 2.2, one obtains the
formula 
\begin{align*}
x \cap  (u \otimes_A \nu) &= \tau^A(x) \otimes_{A \otimes A} (m \otimes_{A} (\tau^A(u) 
\otimes_{A \otimes A} \nu))\\
&= (\tau^A(x) \otimes_{A \otimes A} (m \otimes_{A} (\tau^A(u))) 
\otimes_{A \otimes A} \nu))\\
&= (x \cup u) \otimes_A \nu. 
\end{align*}
Since $x \longmapsto x \cup u$ and $x \longmapsto x
\otimes_A \nu$ are isomorphisms the result follows.

\end{proof}

Additional aspects of this topic, such as the study of
noncommutative algebras which are K-commutative, will be developed in
further work.

\section{Existence of Spanier-Whitehead $K$-Duals}

In this section we show that if $A$ is a separable $C^*$-algebra satisfying the UCT and 
if $K_*(A)$ is finitely generated then $A$ has a Spanier-Whitehead $K$-dual. This result  is 
analogous to the classical theorem that any space of the homotopy type of a finite $CW$-complex 
has a classical Spanier-Whitehead dual.

\begin{Pro} Suppose given a countable $\ZZ/2$-graded abelian group $G_*$.  Then there exists a sequence 
\[
A_0 \to A_1 \to A_2 \to \dots  \to A_n \to \dots
\]
of $C^*$-algebras and $C^*$-maps such that 
\begin{enumerate}
\item The unitalization $A_n^+$ satisfies that $ A_n^+ \cong C(X_n)$ for some finite CW complex $X_n$.
\item Each map
\[
K_*(A_n) \longrightarrow K_*(A_{n+1}) 
\]
is an inclusion. 
\item  There is an isomorphism
\[
\Dirlim K_*(A_n) \cong G_*.
\]
\item Let $A = \Dirlim A_n $.   Then $A \cong C_0(X)$ is a separable commutative $C^*$-algebra in the bootstrap category, and 
\[
K_*(A) \cong G_*  .
\]
\end{enumerate}
\end{Pro}

\begin{proof}
Write $G_*$ as the union of an increasing sequence of finitely generated $\ZZ/2$-graded abelian groups $G_*^n$.  Then apply \cite{UCTMilnor} Theorem 5.1.
\end{proof}

\begin{Thm}  
Suppose that $A$ is a separable $C^*$-algebra that satisfies the UCT and $K_*(A) $ is finitely generated. Then there exists a finite CW-complex (or finite minus a point) $X$ such  
that $A$ is Spanier-Whitehead $K$-dual to $C(X)$ (or $C_0(X \setminus pt)$).
 
\end{Thm}
 
\begin{proof}
Let $Y$ be a finite complex (or finite minus a point) such that $K^*(Y) \cong K_*(A)$.   
 The space $Y$ has a classical Spanier-Whitehead dual; pick one that has a duality map 
$X \times Y \to  S^{2n}$.   Theorem \ref{classical=KK} implies that $C(X)$ and $C(Y)$ are 
Spanier-Whitehead $K$- dual.  Now $A$ and $C(Y)$ are $KK$-equivalent, by the UCT 
\cite{RS}, and so Proposition \ref{T:basic} implies that $A$ and $C(X)$ are Spanier-Whitehead $K$- dual.
\end{proof}
 
\begin{Rem}
If $A$ is separable, satisfies the UCT, but 
 $K_*(A) $ is not finitely generated then separability implies that $K_*(A)$  is   countable, and we may apply the previous result to obtain 
a locally compact space $Y$ such that $K_*(A) \cong K_*(C_0(Y))$.  Then
$A$ and $C_0(Y)$
 are $KK$-equivalent by the UCT. The problem 
now is topological: how do you take the Spanier-Whitehead dual of a compact space that is not of the homotopy type of a finite 
CW-complex? (The situation is analogous to Paschke duality, which we
discuss in Section 7).  It turns out that if $X$ is finite-dimensional
then  one may use {\emph{functional Spanier-Whitehead duals}}
as in \cite{KKS}.
However,  the resulting Spanier-Whitehead dual must be treated as a spectrum rather 
than a space.   In principle one could move to a larger
category at this point, but we refrain. 
\end{Rem}
\begin{Rem}
It is often useful to view the category $\mathcal{KK}$, with objects separable
$C^*$-algebras and with 
morphisms $KK(A,B)$, as analogous to the stable homotopy category
of
countable CW-complexes and stable homotopy classes of maps, $\mathcal{SH}$,
cf. \cite{Meyer, Meyer-Nest}.  In the stable homotopy setting there is a result of
Boardman \cite{Boardman} which implies that the largest full subcategory
of $\mathcal{SH}$ closed under Spanier-Whitehead duality is
that determined by stable homotopy types of finite CW-complexes.  
It is interesting that the results of Section 6 
lead to a noncommutative version of Boardman's theorem.

Let $\mathcal{KK^*}$ be the full subcategory of $\mathcal{KK}$ with objects nuclear
$C^*$-algebras in the bootstrap category.  
The algebras in $\mathcal{KK^*}$ will
satisfy the UCT \cite{RS} and are all KK-equivalent
to $C(X)$ or $C_0(X \setminus pt)$, for $X$ a compact Hausdorff space.   

Let $\mathcal{KK}_F$ be the full subcategory of $\mathcal{KK}^*$ with objects
that have finitely generated K-theory.  

\begin{Pro}
  The category $\mathcal{KK}_F$ is the largest subcategory of $\mathcal{KK}^*$ closed
  under Spanier-Whitehead K-duality. 
\end{Pro}

\begin{proof}
First we note that Theorem 6.2 shows that any object, $A$, in
$\mathcal{KK}_F$ is KK-equivalent to $C(X)$, for $X$ a finite complex.
Thus, $A$ has a dual which is KK-equivalent to $C(Y)$ with $Y$ a
finite complex.  Hence, $\mathcal{KK}_F$ is closed under taking
Spanier-Whithead K-duals.

To complete the proof we must show that any object in $\mathcal{KK^*}$
which has a Spanier-Whitehead K-dual in $\mathcal{KK^*}$
will have finitely generated K-theory, hence will be in $\mathcal{KK}_F$.  
This is proved in \cite{KPW}, Section 4.4(d).  The hypothesis there is
that there is an odd Spanier-Whitehead K-duality, but the proof works
in the even case as well.
\end{proof}

\end{Rem}

\section{Non-existence of Spanier-Whitehead K-Duals}

Not every nice $C^*$-algebra in the bootstrap category has a separable bootstrap $KK$-dual. Here is an example.  
The following proposition is actually an instant consequence of Theorem \ref{T:basic}    but we give a direct proof 
to illustrate what goes wrong.

\begin{Pro}
\label{7.1} Suppose that  $A$ is separable, satisfies the UCT,   $K_0(A) \cong \QQ $   and $K_1(A) = 0$. 
Then  $A$ cannot have a separable Spanier-Whitehead $K$-dual that satisfies the UCT.
\end{Pro}

Note that $A$  may be taken to be  an AF-algebra,  the direct limit of finite dimensional matrix rings, 
and (by the UCT) is unique up to $KK$-equivalence. One may use this $C^*$-algebra to localize 
$K$-theory, so it should not be thought of as bizarre. 

\begin{proof}
Suppose that $A$
has a $K$-dual $DA$ that is separable and satisfies the UCT, so that
$K^0(DA) = \QQ $ and $K^1(DA) = 0$. 
We apply the UCT to $K_1(DA)$,
\begin{equation}
  \label{eq:8}
  0 \to Ext(K^0(DA), \ZZ) \to K_1(DA) \to Hom(K^1(DA),\ZZ) \to 0.  
\end{equation}

But, $Hom(K^1(DA),\ZZ) = 0$ and one has $$K_1(DA) \cong Ext(K^0(DA), \ZZ) =
Ext(\QQ, \ZZ).$$

This leads to a contradiction since it is known that $Ext(\QQ, \ZZ) \cong \RR$,
\cite{Wiegold}\cite{Pext}, but since $DA$ is separable  $K_1(DA)$ is a
countable group.
\end{proof}

\section{Mod-p $K$-theory}

There are two standard constructions of topological mod-p $K$-theory $K_*(A; \ZZ/p)$.  

The first construction,which appears in Schochet \cite{Top4}, is to select a $C^*$ algebra $N$ in the bootstrap 
category with $K_0(N) = \ZZ/p $ and $K_1(N) = 0$, and then for any $C^*$-algebra $A$ define
\[
K_j(A ; \ZZ/n ) = K_j(A \otimes N).
\]

In \cite{Top4} we initially built $N$ from a Moore space (a space whose reduced homology is zero except in one degree, where it is $\ZZ/p$)   and 
then subsequently showed   that any bootstrap choice for $N$ gave an isomorphic theory.  

The second construction, the kernel of which appears in Dadarlat-Loring \cite{Dadarlat-Loring},
 is to select a $C^*$ algebra $N$ in the bootstrap 
category with $K_0(N) = \ZZ/p $ and $K_1(N) = 0$, and then for any $C^*$-algebra $A$ define
\[
K_j(A ; \ZZ/n ) = KK_{j-1}(N, A).
\]

Dadarlat-Loring used a dimension-drop algebra with suitable $K$-theory, but it is clear that any bootstrap choice will work equally well.  Note that 
the dimension-shift comes from the UCT isomorphism 
\[
\ZZ/p  \cong  Ext( K_0(N), K_0(\CC )) \overset{\cong}\longrightarrow KK_1(N, \CC ) .
\]
We were asked by Jeff Boersema whether these two constructions are equivalent.  The second construction is defined on a somewhat smaller category than the first, 
but with that caveat we shall demonstrate that the two constructions are equivalent. 

Let us fix $N$ as above.  Since it is in the bootstrap category we know that $DN$ exists, and   using the  UCT we obtain
\[
K_0(DN) = 0  \qquad \qquad K_1(DN) = \ZZ/p  .
\]
Since $DN$ is also in the bootstrap category, we conclude at once that $SDN$ is $KK$-equivalent to $N$.  Assume that $A$ is separable so that the $KK$-pairing 
is available. Then we have our result:
\[
K_j(A \otimes N) \cong KK_j(\CC, A\otimes N)  \cong KK_j(DN, A) \cong KK_{j-1}(SDN, A) \cong KK_{j-1}(N,A)
\]
and we have proved the following theorem:

\begin{Thm} Suppose that $A$ is separable and that  $N$ is chosen in the bootstrap 
category with $K_0(N) = \ZZ/p $ and $K_1(N) = 0$.  Then the two different constructions of mod-p $K$-theory 
\[
K_j(A ; \ZZ/n ) = K_j(A \otimes N)   \qquad and \qquad   K_j(A ; \ZZ/n ) = KK_{j-1}(N, A)
\]
are naturally equivalent. 
\end{Thm}
\qed

\begin{Rem}  The same argument shows that $K_*(A; G)$ is uniquely defined for any finite abelian group. However if one were dealing with a group such as $\QQ/\ZZ$, 
for instance, then much more care is required. Torsion will be governed by the behavior of the functor $Ext(-, \ZZ)$ and the torsionfree part of this group will bring us to the 
same difficulty illustrated by the case where $K_0(N) = \QQ $. 
\end{Rem}

\begin{Rem}
In the proof of our result we show that $SDN$ is $KK$-equivalent to $N$.  This is actually stronger than Poincar\'e duality, as it corresponds to the statement that the Moore 
space is actually stably homotopy equivalent to its dual.  We may obtain the requisite duality maps in $KK(N\otimes N, \CC )$ and $KK(\CC, N \otimes N)$  by first creating 
the maps at the level of Moore spaces, moving them to $KK$, and then using the $KK$-equivalences.
\end{Rem}

\section{Paschke Duality}
 
We have seen that not every separable $C^*$-algebra has a Spanier-Whitehead $K$-dual, 
even if we make bootstrap hypotheses. In \cite{Paschke},  Paschke   developed a different 
sort of duality that is 
\begin{enumerate}
\item better, because it is defined for every separable $C^*$-algebra;
\item worse, because the resulting dual is in general non-separable, we cannot form the double 
dual, and only one of the two duality maps is present.
\end{enumerate}
After describing the Paschke dual, $\PP (A)$, we discuss the possibility of substituting more tractable 
$C^*$-algebras in place of $\PP (A)$.

These results are due to Paschke \cite{Paschke} as refined by Higson and Roe \cite{HR}.

Let $\HH $ be a separable Hilbert space. Let $\LL(\HH) $ denote the $C^*$-algebra of bounded operators on $\HH$ and $\KK = \KK(\HH)$ 
denote the compact operators.   
Let $\pi : \LL(\HH ) \to \LL (\HH) / \KK  \cong \calk = \calk (\KK) $ be the projection of the bounded operators to the Calkin algebra. 
Suppose that $A$ is a separable, unital $C^*$-algebra with an ample\footnote{A representation $\rho: A \to \LL(H))$ is {\emph{ample}}
if it is non-degenerate and if $\rho(A)\cap \KK = 0$. } representation $ \rho: A \to \LL(H)$.  Define 
\[
\DD_\rho (A) = \{ T \in \LL (H) : \pi (T\rho(a) - \rho(a)T ) = 0 \qquad \forall a \in A \}.
\]
The projection of this algebra in the Calkin algebra is  $\PP (A)  = \pi (\DD_{\rho}(A))$, the {\emph{Paschke dual}} of $A$. 
Since $\PP (A) $ is independent of the choice of ample
representation by Voiculescu's Theorem \cite{Voiculescu}, we shall
drop $\rho$ from the notation.  
In general $\PP (A)$ is unital, but it is typically neither separable nor nuclear.
Paschke's theorem is the following (\cite{Paschke}, Theorem 2 ).  
\begin{Thm}  Let $A$ be a separable, unital $C^*$-algebra with an
  ample representation on $\LL(\HH)$. Then one has
that
\[
K_0(\PP (A) ) \cong Ext^1(A) 
\]
and hence, if $A$ is nuclear, that 
\[
K_0(\PP(A)) \cong K^1(A).
\]
and similarly for $K_1$.  
\end{Thm}
 
We
note that there is a canonical 
$*$-homomorphism 
\[
\Psi : A \otimes \PP (A) \longrightarrow \calk 
\]
given by 
\[
\Psi (x\otimes y) = \pi (\rho(x))y
\]
which is well-defined because $\pi\rho (A) $ commutes with each element of $\PP (A)$. 
The Kasparov groups $KK_*(A\otimes \PP(A), \calk) $ are defined and so we have 
\[
\nu = [\Psi] \in KK_0(A\otimes \PP(A), \calk ).
\]
Although the full Kasparov
product is not available (since $A\otimes \PP (A)$ is not separable), the slant product with 
the map $\Psi $ still makes sense and gives us a well-defined map

\[
K_0(  \PP(A)   )    \xrightarrow { (-)\otimes_{      \PP A}  \nu  } KK_0(A ,  \calk) 
  \overset{\delta}\longrightarrow
  KK^1(A, \KK ) \equiv K^1(A)
\]
which Paschke shows is an isomorphism.   Thus Paschke's duality result is a  one-sided 
duality.

The simplest case is actually of interest.  Take $A = \CC $.  Then $\PP(A) = \calk $,  $\Psi = 1_\calk $, 
\[
\nu = [1_\calk]    \in KK_0(\calk, \calk ).
\]
\[
\delta :   KK_0(\calk, \calk )  \longrightarrow KK_1(\calk, \KK )   
\]
 and  the UCT index map 
\[
\gamma _\infty :  KK_1(\calk, \KK )   \overset{\cong}\longrightarrow    Hom( K_1(\calk ), K_0(\KK)) \cong \ZZ 
\]
gives the Paschke
isomorphism 
\[
K_1(\calk ) \overset{\cong}\longrightarrow  K^0(\CC ) \cong \ZZ  .
\]
If we regard $K^0(\CC ) \cong K^1(S\CC ) = K^1(C_0(\RR ))$ then we have a way to realize a map in the other direction. 
Let 
\[
\tau :  C(S^1) \longrightarrow \calk 
\]
be the map that takes $z$ to the image of the adjoint of the
unilateral shift $U^*$. This map classifies the extension
 \footnote{ This
  is the storied extension that started the BDF work on the classification of essentially normal operators.}

\[
0 \to \KK \longrightarrow C^*\{ \KK , U^*, I\}  \longrightarrow C(S^1) \to  0
\]
Restrict $\tau $ to $C_0(\RR)$.    We then have the pullback diagram
\[
\begin{CD}
0   @.    0   \\
@VVV  @VVV   \\
\KK  @>>>  \KK   \\
@VVV      @VVV   \\
\mathcal{E}  @>>>  \mathcal{L}(\mathcal{H})  \\
@VVV   @VVV  \\
C_o(\RR) @>\tau >>  \calk  \\
@VVV   @VVV   \\
0   @.   0
\end{CD}
\]
The right column generates a (very!) canonical extension $\Upsilon \in Ext(\calk, \KK)$  and 
\[
[\tau] = \tau ^*(\Upsilon) \in Ext( C_o(\RR ) , \KK) \cong KK^1( C_o(\RR ) , \KK)  .
\]
Further, 
\[
\gamma _\infty ([\tau]) : K_1(C_o(\RR )) \overset\cong\longrightarrow K_0(\KK)  
\]
and this map is in a sense the inverse to the Paschke 
 isomorphism.  This example is the basis for our hope for 
the Conjecture at the end of Section 1.

The Paschke dual is not a Spanier-Whitehead $K$-dual, in general, for several related 
reasons. It is usually (perhaps always) non-separable, its $K$-theory is not necessarily 
finitely generated and may well be uncountable even for $A$ an AF-algebra, and there does 
not seem to be a duality class $\CC \to A\otimes \PP(A)$.  We discuss what can be done 
in future sections.

\section{$C^*$-substitutes I:  $K_*(A)$ countable}
 
In this section we show that if $A$ is a (nuclear) $C^*$-algebra with $K_*(A)$ countable 
then there exists a separable  (nuclear) sub-$C^*$-algebra $\theta A \subseteq A $  which 
is weak $K$-equivalent to $A$.

\begin{Def} A $*$-homomorphism $f: A \to  B $ is a {\emph{weak
      $K$-equivalence}} if the induced map $f_*: K_*(A) \to K_*(B)$ is
  an isomorphism, \cite{RS}.
\end{Def} 
Note that if $A$ satisfies the UCT   then
 a weak $K$-equivalence $f: A \to B$  lifts to a $KK$-class $\mu \in KK_0(A,B)$. 
If $B$ is also in the UCT class then this class may be chosen to be $KK$-invertible, so that 
$A$ is $KK$-equivalent to $B$. 

In the other direction, if $\mu \in KK_0(A,B)$ is an invertible class then it induces an isomorphism 
$\mu/ : K_*(A) \overset{\cong}\to K_*(B)$ but it does not necessarily arise from a map $A \to B$. 
Here are two examples:

\begin{enumerate}
\item  $M_3(\CC )$ and $M_2(\CC )$ are $KK$-equivalent but there is no map $M_3(\CC) \to M_2(\CC) $ inducing this equivalence.  

\item   $C(\CC P^2)$ and $C(S^2\vee S^4)$  are $KK$-equivalent, but there is no map 
of spaces that can induce this equivalence, since $K^*(\CC P^2)$ and $K^*(S^2 \vee S^4)$ are
not isomorphic as rings.  The authors learned the cohomology version
of this example as students from an unpublished paper of Steenrod,
since published as \cite{Steenrod}.

\end{enumerate}

\begin{Pro}\label{jonathan}
Let $A$ be a $C^*$-algebra and suppose that $K_*(A)$ is countable.  Then there exists a separable subalgebra $F$ of $A$ such that the inclusion map $\iota : F \to A$ induces a surjection $\iota _*: K_*(F)\to K_*(A)$. 
\end{Pro}

 \begin{proof}   Since $K_*(A) $ is countable we may list a countable family of projections 
and unitaries that generate $K_0$ and $K_1$ respectively. Each of these lies in some 
$A\otimes M_n$.  Take the (countable) collection of elements of $A$ that are the matrix entries 
of this family and let $F$ be the subalgebra of $A$ that they generate.  Then it is clear that 
the map $\iota _* : K_*(F) \to K_*(A)$ is surjective. 
\end{proof}

The map $\iota _* : K_*(F) \to K_*(A) $  probably is not injective in general. To remedy this problem we use the following construction, due to Ilan Hirshberg.

\begin{Lem} 
Suppose given  a $C^*$-algebra $A$ and a $C^*$-subalgebra $\iota : B  \to  A$. Suppose
$x \in K_0(B)$ and $\iota _*(x) = 0$. Then there  are elements $\{a_1, \dots a_n \} $ of $A$ with the
property that if $B'$ is the $C^*$-subalgebra generated by $B \cup \{a_1, \dots, a_n\}$  with 
inclusion map $\iota ' : B \to B' $, then  
$\iota _*' (x)  = 0 \in K_0(B')$. 
\end{Lem}

\begin{proof}
Represent $x = [p] - [q]$ where $p$ and $q$ are projections in matrix rings over $B$.  
The fact that $\iota _*(x) = 0 $ means that we have
\[
[p] - [q] = [t] - [t]
\]
for some trivial projection $t$. Unraveling this leads us to the equation
\[
upu^* \oplus h = w(vqv^*\oplus h)w^*
\]
for some unitaries $u, v, w$ and some projection $h$, where $u, v, w,$ and $h$ lie in matrix rings over
$B$. Take the set $\{ a_1, \dots a_n \}$ to be the (finite!) collection of matrix coefficients 
in the matrices $u, v, w, h$. Then it is obvious that the same calculations that took place in $A$ 
can take place in $B'$, and so $\iota _*' (x) = 0$ as desired.
\end{proof}

\begin{Lem}
Suppose given  a $C^*$-algebra $A$ and a $C^*$-subalgebra $\iota : B  \to  A$. Suppose
$x \in K_1(B)$ and $\iota _*(x) = 0$. Then there  are elements $\{a_1, \dots a_n \} $ of $A$ with the
property that if $B'$ is the $C^*$-subalgebra generated by $B \cup \{a_1, \dots, a_n\}$  with 
inclusion map $\iota ' : B \to B' $, then then 
$\iota _*' (x)  = 0 \in K_1(B')$. 
\end{Lem}

\begin{proof} Represent $x$ by $u \in U_n(B)$.  The fact that $\iota _*(x) = 0$
translates into the existence of a continuous path of unitaries $u_t \in U_{n+k}(A) $
for some $k$ such that $u_0 = u\oplus I $ and $u_1 = I$.  Pick a finite sequence 
of elements $a_j$ on this path with $a_0 = u_0$, $a_n = I$, and with 
 the property that $|a_{j}^{-1}a_{j+1}| < 1$.   Then we may construct a path in $U_{n+k}(B')$ 
connecting these same elements, and hence $u\oplus I $ is in the path component 
of the identity of $U_{n+k}(B')$, showing that $\iota _*'(x) = 0.$
\end{proof}

\begin{Lem}\label{ilan}
 Suppose given  a $C^*$-algebra $A$ and a $C^*$-subalgebra $\iota : B  \to  A$ with 
associated map 
\[
\iota _* : K_*(B) \longrightarrow K_*(A) .
\]
Suppose that $Ker(\iota _*)$ is countable.  Then there exists a countable number of elements $\{a_j\}$ of $A$ 
such that if we let $B'$ denote the $C^*$-algebra generated by $B$ and by the 
$\{a_j\}$ and  let $\iota ' : B \to B'$ denote the inclusion, then 
\[
Ker (\iota _*) = Ker (\iota _*')   .
\]
If $Ker (\iota _*)$ is finitely generated then only a finite number of additional elements 
are needed. 
\end{Lem}

\begin{proof} This follows immediately from the previous two lemmas-  we simply choose 
generators for $Ker (\iota _*)$ and kill them off by adding all of the needed additional 
elements at once.
\end{proof}

\begin{Thm} (I. Hirshberg)\label{ilan}  Suppose that $A$ is a $C^*$-algebra with $K_*(A)$ countable. Then there exists an ascending
sequence of separable sub-$C^*$-algebras of $A$
\[
F_1 \subset F_2 \subset F_3 \subset \dots 
\]
with coherent inclusion maps $\iota _n : F_n \to A$ such that
each map $\iota _{n*} : K_*(F_n) \to K_*(A) $ is surjective. 
Let  $\theta A = \Dirlim F_j $. 
 Then $\theta A$ is separable and  the induced inclusion  map 
$
\iota : \theta A  \longrightarrow A
$
yields an isomorphism 
\[
\iota _* : K_*(\theta A ) \overset{\cong}\longrightarrow K_*(A).
\]

\end{Thm}

\begin{proof} We use Lemma \ref{jonathan}  to construct $F_1$ together with 
the map 
\[
\iota _1 : K_*(F_1) \to K_*(A) 
\]
 which induces a surjection in $K$-theory. 
Then repeatedly use Lemma \ref{ilan} to construct the higher $F_n$.  This gives us 
an ascending 
sequence of sub-$C^*$-algebras 
\[
F_1 \subset F_2 \subset F_3 \subset \dots 
\]
with coherent inclusion maps $\iota _n : F_n \to A$  and
\[
Ker(\iota _{n*}) \subseteq Ker [K_*(F_n) \to K_*(F_{n+1}) ].
\]

  Since the map 
\[
\iota _1 : K_*(F_1) \to K_*(A) 
\]
 is surjective the induced map 
\[
\iota _* : K_*(\theta A) \to K_*(A) 
\]
 is surjective.  Finally,  we claim that $\iota _* $ 
is injective, and hence an isomorphism.  Suppose that $\iota _*(y) = 0$. 
Then the class $x$ must arise in some $K_*(F_n)$ with $\iota _{n*}(x) = 0$. 
But then $x \in Ker(\iota _{n*})$ and so $x = 0 \in K_*(F_{n+1})$.  Thus 
$x = 0 \in K_*(A)$ and the proof is complete.

\end{proof}

\begin{Cor} In Theorem \ref{ilan},  if $A$ is nuclear then $\theta A$ may be constructed to be separable and nuclear.
\end{Cor}

\begin{proof} 
We construct inductively an increasing sequence of separable subalgebras $F_n$ of $A$, as follows. 
$F_1$ will be the one described as in the proof of  Theorem \ref{ilan}. 

Choose a countable dense subset of the unit ball of $F_1$, call it $S_1$. Regard $S_1$ as a sequence.  Since $A$ is nuclear, we can find completely positive contractions $\psi:A \to M_k$, $\omega:M_k \to A$ for some $k$ such that
\[
 ||\omega(\psi(a)) - a || < 1,
\]
 where $a$ is the first element in $S_1$.

Now, let $F_2$ be the subalgebra generated by $F_1$, all the elements which are added according to the proof above, and the image of the map $\psi$  (which is finite dimensional, so it is still separable). Now choose a dense subset $S_2$ of the unit ball of $F_2$, again ordered as a sequence.

Suppose we constructed 
\[
F_1 \subset F_2 \dots \subset F_n,
\]
 along with dense sequences $S_1,S_2,\dots S_n$ of the respective unit balls. Pick the first $n$ elements of each of the sets $S_1,...,S_n$, and call this set $S$ (it has at most $n^2$ elements). Pick completely positive contractions $\psi:A \to M_j$, $\omega:M_j \to A$ for some $j$ such that
\[
||\omega(\psi(a)) - a || < 1/n
\]
 for all $a \in S$. Now, modify the definition of $F_{n+1}$ to be generated by the elements as in the proof of the Theorem along with $\omega(M_j)$.

The closure of the union, $\theta A$, is now nuclear. To see this, one
needs to verify that $\theta A$ has the Completely Positive
Approximation Property, CPAP (\cite{L}, p. 170).
One may start with a finite subset $X$ of the unit ball and an
$\epsilon >0$. It can be assumed that $X$ is in the union of the
$S_n$'s, since they are dense. If one goes far enough out in the
sequence of inclusions (e.g. find an $N$ so that $1/N<\epsilon $ and
$X$ is contained in the union of the first $N$ elements of each of
$S_1,...,S_N$), then the maps $\psi $ (restricted to $F$) and $\omega$
(whose image is in $\theta A$) which were used to define $F_{n+1}$ now
witness the CPAP for the finite set $X$ to within tolerance $\epsilon$. (None of the $F_n$'s need be nuclear themselves, but the union is.)
 
\end{proof}
 
 \begin{Rem} Our construction of  the subalgebra $\theta A $ in $A$ in Theorem \ref{ilan} involves 
many choices and hence there is no reason to think that $\theta A$ is uniquely defined. At best one 
might hope that any two choices would be $KK$-equivalent. This would follow at once if $\theta A$ 
satisfied the UCT .
\end{Rem}

\section{$C^*$-substitutes II: bootstrap entries}

We would like to know that every $C^*$-algebra $A$ has a commutative  (or at least a bootstrap) 
$C^*$-algebra that is weakly $K$-equivalent to it.  In the previous section we showed that if 
$K_*(A)$ is countable then 
up to weak $K$-equivalence we can replace $A$ by a separable subalgebra.  If $K_*(A)$ is 
uncountable then obviously any substitute will be non-separable, but still we could hope for 
commutativity. In this section we demonstrate that it is almost possible to have a commutative 
substitute. 

 If  $A$ satisfies the UCT then $A$ is $KK$-equivalent to a commutative $C^*$-algebra $C$, 
but the invertible $KK$-elements that link them are not necessarily implemented by maps $C \to A$
or vice versa. In this section we prove that if $A$ satisfies the UCT then 
there exists a $2$-step solvable  (hence bootstrap) 
$C^*$-algebra $\beta A$ and an auxiliary $C^*$-algebra $T$ together with maps 
$\beta A \to T \leftarrow S^3A$   that are weak $K$-equivalences.

The following lemmas and the theorem are variants of the  original
argument of the second author, (\cite{Top2}, Lemma 3.1) 
used in the proof of the K\"unneth formula and also
the revised argument due to Blackadar (\cite{Blackadar}, Theorem 23.51).

\begin{Lem} Suppose that $K_1(A) \cong \ZZ ^s $ with $s$ finite,
  countably infinite, or uncountable . Then there exists a map 
\[
f:  \oplus_s C_0(\RR) \longrightarrow A \otimes \KK 
\]
such that the induced map 
\[
f_* : K_1(  \oplus_s C_0(\RR)  )  \longrightarrow K_1(A\otimes \KK )
\]
is an isomorphism (and the induced map on $K_0$ is trivial). 
\end{Lem}

\begin{proof} 
Choose unitaries  $\{u_1, u_2, \dots  \} \subset (A\otimes \KK)^+ $ which represent  a minimal set of generators of $K_1(A)$. Without loss of generality 
we may take these generators to be mutually orthogonal. They induce the obvious map 
\[
\oplus _s C(S^1) \longrightarrow (A\otimes \KK)^+
\]
which is an isomorphism on $K_1$.  Define $f$ to be the restriction of this map to $ \oplus_s C_0(\RR) $; it factors through $A\otimes \KK$ 
and the result follows.
\end{proof}

\begin{Lem}  
Suppose that $K_0(A) \cong \ZZ ^r $ with $r$ finite, countably infinite, or uncountable. Then there exists a map 
\[
f:  \oplus_s C_0(\RR) \longrightarrow SA \otimes \KK 
\]
such that the induced map 
\[
f_* : K_1(  \oplus_s C_0(\RR)  )  \longrightarrow K_1(SA\otimes \KK )
\]
is an isomorphism. Suspending, we obtain  a map $g$, 
\[
g:  \oplus _s C_0(\RR ^2) \cong S(\oplus_s C_0(\RR)) \longrightarrow S^2A \otimes \KK 
\]
such that the induced map 
\[
g_{*}: K_0( \oplus _s C_0(\RR ^2)) \longrightarrow K_0(S^2A \otimes \KK )\cong K_0(A) 
\]
is an isomorphism, and the induced map on $K_1$ is trivial.

\end{Lem}

Combining these two lemmas gives us the desired result. 

\begin{Thm}
Suppose that $A$ is a $C^*$-algebra with $K_*(A) $ free abelian. Then 
\begin{enumerate}
\item
There is a commutative 
$C^*$-algebra 
$C$ which is a direct sum of copies of $C_0(\RR ^2 )$ and  $C_0(\RR^1 )$  and a map 
\[ 
h : C \longrightarrow SA \otimes \KK 
\]
such that the induced map 
\[
h_* : K_*(C) \longrightarrow   K_*(SA\otimes  \KK )  \cong K_{*-1}(A)
\]
is an isomorphism. 

\item     
     Suspending, 
there is a a commutative $C^*$-algebra 
$SC$ which is a direct sum of copies of $C_0(\RR ^3 )$ and  $C_0(\RR^2 )$  and a map 
\[ 
h : SC \longrightarrow S^2A \otimes \KK 
\]
such that the induced map 
\[
h_* : K_*(SC) \longrightarrow   K_*(S^2A\otimes  \KK )  \cong K_*(A)
\]
is an isomorphism.

\end{enumerate}

\end{Thm}

\begin{proof} For the first statement, take  
\[
  Sf \oplus g  :   \big(    \oplus_s C_0(\RR^2 ) \big) \oplus \big(        \oplus_s C_0(\RR^1)  \big) \longrightarrow SA \otimes \KK   .
\]
For the second part, simply suspend.

\[
h = S^2f \oplus Sg  :   \big(    \oplus_s C_0(\RR^3 ) \big) \oplus \big(        \oplus_s C_0(\RR^2)  \big) \longrightarrow S^2A \otimes \KK   .
\]
 \end{proof}

Here is a restatement of the previous results couched in terms of $\beta A$. 

\begin{Thm}  Suppose given a  $C^*$-algebra $A$ with $K_*(A)$ free abelian.
Then there exists a $C^*$-algebra $\beta A$ with the following properties:
\begin{enumerate}
\item There is a map $h:  \beta A \longrightarrow   S^2A\otimes \KK $ which induces an 
isomorphism 
\[
h_* :   K_*(\beta A) \overset{\cong}\longrightarrow K_*(A)
\]
so that $\beta A$ is weakly $K$-equivalent to $S^2A \otimes \KK $. 
\item If $A$ is separable (or, more generally, if $K_*(A)$ is countable) then 
$\beta A$ is separable.
\item $\beta A$ is commutative and is the direct sum of copies of $C_0(\RR ^3)$ and $C_0(\RR ^2 )$ . 
\item If $K_*(A)$ is countable then $\beta A$ is in the bootstrap category. 
\end{enumerate}

\end{Thm}

\begin{proof} Take $\beta A = SC$ as above.
\end{proof}

If $K_*(A)$ is not free abelian then our results are unfortunately not so neat. Here is what happens:

\begin{Thm} 
\label{oldapprox}
 Let $A$ be a $C^*$-algebra. Then there exists a $C^*$-algebra $\beta A$ with the following 
properties:
\begin{enumerate}
\item 
\[
K_*(\beta A) \cong K_*(A).
\]
\item If $A$ is separable then $\beta A$ is separable. 
\item If $K_*(A)$ is countable then $\beta A$ is in the bootstrap category. 
\item $\beta A$ fits into a short exact sequence of the form
\[
0 \to C_0(X_1)\otimes\KK  \to \beta A  \to C_0(X_2)\otimes\KK  \to 0
\]
where $X_j$ consist of disjoint unions of lines, planes, and their suspensions. Thus $\beta A$ is a solvable
$C^*$-algebra.  If $K_*(A)$ is countable (resp. finitely generated) then the $X_j$ are disjoint unions 
of countable (resp. finite) number of components.

\item There exists an auxiliary $C^*$-algebra $T$ and maps
\[
\beta A \overset{h}\longrightarrow T \overset{j}\longleftarrow S^3A
\]
with the following properties:
\begin{enumerate}
\item The map $h$ is a weak $K$-equivalence.
\item The map $j$ is the inclusion of an ideal, and $T/S^3A $ is a contractible $C^*$-algebra. In particular, 
$j$ is also a weak $K$-equivalence. 
\end{enumerate}
 \end{enumerate}

\end{Thm}

\begin{Rem} It is interesting to compare the properties of $\beta A$ with the properties 
of $\theta A$ in Theorem \ref{ilan} under the assumption that $K_*(A)$ is countable.  On 
the one hand, $\beta A$ is a better behaved approximation for $A$ than $\theta A$ because it is  solvable and satisfies the UCT. 
On the other hand, the inclusion $\theta A \to A$ is a weak $K$-equivalence, whereas for $\beta A$
the best we can do is a sequence of $K$-equivalences 
\[
\beta A \overset{h}\longrightarrow T \overset{j}\longleftarrow S^3A,
\]
one of which points in the wrong direction! 
\end{Rem}

\begin{proof}
We may assume without loss of generality that $A$ is stable, i.e. $A \cong A\otimes\KK$. 
The case where $K_*(A)$ is free abelian is covered by the previous proposition. Consider the general case.    There is a stably commutative $C^*$-algebra $N $ with $K_*(N)$ free abelian,
 and a map 
\[
f: N \to SA\otimes\KK 
\] 
inducing a surjection
\[
K_*(N) \overset{f_*}\to K_*(SA)  \to 0.
\]
 
Form the mapping cone sequence
\[
0 \to S^2A  \longrightarrow Cf  \overset{\pi}\longrightarrow N \to 0.
\]
We may assume that $Cf$ is stable.
The associated $K$-theory sequence corresponds via the suspension isomorphism to the sequence
\[
0 \longrightarrow K_*(Cf) \longrightarrow K_*(N) \overset{f_*}\longrightarrow K_*(SA) \to 0 .
\]
Thus $K_*(Cf)$ is free abelian, and the sequence above is a free resolution of $K_*(SA)$. 
 Proposition \ref{oldapprox} tells us that there is a stably commutative $C^*$-algebra $M$
and a weak $K$-equivalence 
$g: M \to SCf$ with associated mapping cone sequence
\[
0 \to S^2Cf \to Cg \to M \to 0.  
\]
Note that $K_*(Cg) = 0$ since $g$ is a weak $K$-equivalence, and hence there is a natural diagram
\[
\begin{CD}
0 @>>>    K_*(SCf) @>(S\pi)_*>>    K_*(SN)     @>f_*>>   K_*(S^2A)    @>>>    0 \\
@.  @Ag_*A\cong A     @A1AA     @A1AA       \\
0 @>>>   K_{*}(M) @>(S\pi)_*g_*>>    K_*(SN) @>f_*>>   K_*(S^2A) @>>>    0
\end{CD}
\]

Now consider the composition 
\[
M \overset{g}\longrightarrow SCf \overset{S\pi}\longrightarrow SN
\]
and define the mapping cone of the composition by
\[
\beta A = C((S\pi)g) .
\]

The mapping cone sequence takes the form
\[
0  \to S^2N \to \beta A \to M \to 0
\]
and fits into a natural diagram
\[
\begin{CD}
0 @>>>  S^2N  @>>>   \beta A   @>>>    M   @>>> 0  \\
@.   @VV1V   @VVhV   @VVgV   @.  \\
0  @>>>    S^2N  @>>>  C(S\pi) @>>>    SCf   @>>>  0
\end{CD}
\]
Applying $K$-theory to this diagram yields the following diagram, with exact rows:
\[
\begin{CD}
@>>>   K_*(\beta A)   @>>>   K_*(M)  @>{(S\pi)_*g_*}>>  K_{*-1}(S^2N) @>>>  K_{*-1}(\beta A) @>>>   \\
   @.     @VVV   @VVg_*V    @VV\cong V    @VVh_*V    \\
@>>>   K_*(C(S\pi))   @>>>   K_*(SCf)  @>{(S\pi)_*}>>  K_{*-1}(S^2N) @>>>  K_{*-1}(C(S\pi)) @>>>
\end{CD}
\]
The map $g_*$ is an isomorphism and  the map  $(S\pi)_*$ is mono, and so the diagram simplifies 
to the diagram
   \[
\begin{CD}
0   @>>>   K_*(M)  @>{(S\pi)_*g_*}>>  K_{*-1}(S^2N) @>>>  K_{*-1}(\beta A) @>>>  0 \\
   @.     @VVg_*V   @VV\cong V    @VVh_*V       \\
0   @>>>   K_*(SCf)  @>{(S\pi)_*}>>  K_{*-1}(S^2N) @>>>  K_{*-1}(C(S\pi)) @>>> 0
\end{CD}
\]
The Five Lemma implies that the map 
\[
h_* : K_*(\beta A) \to K_*(C(S\pi))
\]
is an isomorphism. 

Recall that the map $\pi : Cf \to N$ fits into the sequence
\[
0 \to S^2A \to Cf \overset{\pi}\longrightarrow N \to 0.
\]
Suspending yields the exact sequence
\[
0 \to S^3A \to SCf \overset{S\pi}\longrightarrow SN  \to 0.
\]
Since $S\pi $ is surjective, its cone sequence fits into the following
diagram, by \cite{Top3} Proposition 2.3,
\[
\begin{CD}
@.@.  0   \\
@.@.  @VVV   \\
@.@.  S^3A  \\
@.@.  @VVjV   \\
0 @>>>   SN  @>>>  C(S\pi)   @>>>  SCf   @>>>  0  \\
@.  @.   @VVV   \\
@.@.  CN  \\
@.@.  @VVV   \\
@.@.  0   
\end{CD}
\]
where $CN$ denotes the cone on $N$, which is contractible.  In particular, the natural map $j: S^3A \to C(S\pi)$ 
is a weak $K$-equivalence.  Let $T = C(S\pi)$ for brevity.

To summarize, we have constructed $C^*$-maps 
\[
S^3A \overset{j}\longrightarrow T  \overset{h} \longleftarrow  \beta A
\]
which are both weak $K$-equivalences. The $C^*$-algebra $\beta A$ fits in a sequence of the form
\[
0 \to C_0(X_1)\otimes\KK  \to \beta A  \to C_0(X_2)\otimes\KK  \to 0
\]
and is hence two-step solvable.  This completes the proof.
\end{proof}

\begin{Rem}  If $K_*(A)$ is countable then $\beta A$ may be chosen to be in the bootstrap category, and then the UCT implies that any 
two choices will be $KK$-equivalent.  If $K_*(A)$ is free abelian then its maximal ideal space is uniquely determined up 
to homeomorphism, simply by counting components. 
\end{Rem}




\end{document}